\documentclass[english,reqno]{amsart}

\usepackage{amsmath, appendix, ulem}
\usepackage[nobysame]{amsrefs}
\usepackage{amssymb, color}
\usepackage[margin=1in]{geometry}
\usepackage{mathrsfs}
\usepackage{graphicx}
\usepackage{subfig}
\usepackage{float}
\usepackage{epsf}
\usepackage{hyperref}
\usepackage{titletoc}

\usepackage{subeqnarray}
\usepackage{cases}
\graphicspath{{../Figures/}}

\numberwithin{equation}{section}

\newtheorem{lem}{Lemma}[section]
\newtheorem{thm}{Theorem}[section]
\newtheorem{proposition}[thm]{Proposition}
\newtheorem{cor}[thm]{Corollary}
\theoremstyle{remark}

\newtheorem{definition}[thm]{Definition}
\newtheorem{assum}{Assumption}

\newcommand{\nn}{\nonumber}
\newcommand{\R}{{\mathbb R}}

\newcommand{\dx}{ \, {\rm d} x}
\newcommand{\dy}{ \, {\rm d} y}

\renewcommand{\tilde}{\widetilde}
\renewcommand{\hat}{\widehat}

\newcommand{\mc}[1]{\mathcal{#1}}

\newcommand{\CR}{\mathcal{C}([0,T];\mathbb{R}^d)}

\newcommand{\EE}{\mathbb{E}}
\newcommand{\TE}{\mathcal{E}}
\newcommand{\RR}{\mathbb{R}}
\newcommand{\NN}{\mathbb{N}}
\newcommand{\PP}{\mathbb{P}}

\newcommand{\OV}{\overline{V}}
\newcommand{\OX}{\overline{X}}

\newcommand{\ud}{\,\mathrm{d}}

\newcommand{\la}{\langle}
\newcommand{\ra}{\rangle}

\newcommand{\norm}[1]{\left\lVert#1 \, \right\rVert}

\DeclareMathOperator\diag{diag}

\author{Hui Huang}
\address{Department of Mathematics and Statistics, University of Calgary, 2500 University Dr NW, Calgary, AB T2N 1N4, Canada}
\email{hui.huang1@ucalgary.ca}
\author{Jinniao Qiu}
\address{Department of Mathematics and Statistics, University of Calgary, 2500 University Dr NW, Calgary, AB T2N 1N4, Canada}
\email{jinniao.qiu@ucalgary.ca}

\thanks{H. H. is partially supported by the Pacific Institute for the Mathematical Sciences (PIMS) postdoc fellowship. J. Q. is partially supported  by the National Science and Engineering Research Council of Canada (NSERC) and by the start-up funds from the University of Calgary.}
\begin{document}
\title[Mean-field limit for CBO]{On the mean-field limit for the consensus-based optimization}
\maketitle
\begin{abstract}
	This paper is concerned with the large particle limit for the consensus-based optimization (CBO), which was postulated in the pioneering works \cite{carrillo2018analytical,PTTM}. In order to solve this open problem, we adapt a compactness argument by first proving the tightness of the empirical measures $\{\mu^N\}_{N\geq 2}$ associated to the particle system and then verifying that the limit measure $\mu$ is the unique weak solution to the mean-field CBO equation. Such results are extended to the model of particle swarm optimization (PSO).
\end{abstract}
{\small {\bf Keywords:} Consensus-based optimization, particle swarm optimization, propagation of chaos,  tightness, weak convergence.}

\section{Introduction} 
The global optimization of a potentially nonconvex nonsmooth cost function is of great interests in various areas such as economics, physics,  and artificial intelligence. In the sequel, we consider the following optimization problem
\begin{equation}\label{optimization}
x^*\in\mbox{argmin}_{x\in \RR^d}\TE(x)\,,
\end{equation}
where $\TE(x):~\RR^d\to \RR$ is a given continuous cost function, which one wishes to minimize. Many methods have been designed to tackle this kind of problems. The present paper is  in particular concerned with the methods of so-called \textit{metaheuristics} 
\cite{aarts1989simulated,Back:1997:HEC:548530,Blum:2003:MCO:937503.937505,Gendreau:2010:HM:1941310} which provide empirically robust solutions to tackle hard optimization problems with fast algorithms. Metaheuristics are  methods that orchestrate an interaction between local improvement procedures and global/high level strategies and combine random and deterministic decisions,
to create a process capable of escaping from local optima and performing a robust search of a solution space. Noble examples of metaheuristics include Simplex Heuristics \cite{nelder1965simplex}, Evolutionary Programming \cite{fogel2006evolutionary}, Genetic Algorithms \cite{holland1992adaptation}, Particle Swarm Optimization \cite{kennedy1995particle}, Ant Colony Optimization \cite{dorigo2005ant}, and Simulated Annealing \cite{aarts1989simulated}. Recently a new type of metaheuristics was proposed in \cite{carrillo2018analytical,PTTM}, which is referred to as consensus-based optimization (CBO) method.

The consensus-based optimization takes advantage of an interacting $N$-particle system $\{(X_t^{i,N})_{t\geq 0}\}_{i=1}^N$, which is described by a system of stochastic differential equations (SDEs)
\begin{align}\label{particle}
dX_t^{i,N}=-\lambda(X_t^{i,N}-X_\alpha(\mu_t^N))dt+\sigma D(X_t^{i,N}-X_\alpha(\mu_t^N))dB_t^i\,,
\end{align}
where $\lambda,\sigma>0$, 
\begin{equation}
X_\alpha(\mu_t^N)=\frac{\int_{\R^d}xe^{-\alpha\TE(x)}\mu_t^N(dx)}{\int_{\R^d}e^{-\alpha\TE(x)}\mu_t^N(dx)}\quad \mbox{ with }\mu_t^N=\frac{1}{N}\sum_{i=1}^{N}\delta_{X_t^{i,N}}\,,
\end{equation}
 and $\{(B_t^i)_{t\geq0}\}_{i=1}^N$ are $N$ independent $d$-dimensional Wiener processes. We also use the following notation for the diagonal matrix
$$D(X_t):=\mbox{diag}\{(X_t)_1,\dots,(X_t)_d\}\in\RR^{d\times d}\,,$$ 
where $(X_t)_k$ is the $k$-th component of $X_t$.
 The choice of the weight function $$\omega_\alpha^\TE(x):=\exp(-\alpha\TE(x)),$$ comes from  the  well-known Laplace's principle \cite{miller2006applied,Dembo2010}, a classical asymptotic method for integrals, which states that for any probability measure $\mu\in\mc{P}( \RR^d )$, there holds
\begin{equation}\label{lap_princ}
\lim\limits_{\alpha\to\infty}\left(-\frac{1}{\alpha}\log\left(\int_{ \RR^d }\omega_\alpha^\TE(x)\mu(dx)\right)\right)=\inf\limits_{x \in \rm{supp }(\mu)} \TE(x)\,.
\end{equation}
Thus for $\alpha$ large enough, one expects that $${X}_{\alpha}(\mu_t^{N})\approx\mbox{argmin }\{\TE(X_t^{1}),\dots,\TE(X_t^{N})\}\,,$$
which means that ${X}_{\alpha}(\mu_t^{N})$ is a global best location at time $t$.   It has been proved that CBO can guarantee global convergence under suitable assumptions \cite{fornasier2021consensus} and it is a powerful and robust method to solve many interesting non-convex high-dimensional optimization problems in machine learning \cite{CJ,FHPS2}. By now, CBO methods have also been generalized  to optimization over manifolds \cite{FHPS1,FHPS2,kim2020stochastic,fornasier2021anisotropic} and several variants have been explored, which use additionally, for instance, personal best information \cite{totzeck2020consensus}, binary interaction dynamics \cite{benfenati2021binary} or connect CBO with  Particle Swarm Optimization \cite{grassi2020particle,cipriani2021zero}. The readers are referred to \cite{totzeck2021trends} for a comprehensive review on the recent developments of the CBO methods.

Because of the nonlinear and nonlocal term $X_\alpha(\mu^N)$, the conventional method (see e.g. \cite{sznitman1991topics,fetecau2019propagation}) for the mean-field limit does not work here and the pioneering CBO works  \cite{carrillo2018analytical,PTTM} \textit{postulated} the large particle limit (as $N\to \infty$) of the system \eqref{particle} towards the Mckean process
\begin{equation}\label{Mckean}
d\OX_t=-\lambda(\OX_t-X_\alpha(\mu_t))dt+\sigma D(\OX_t-X_\alpha(\mu_t))dB_t,
\end{equation}
where
\begin{equation}
X_\alpha(\mu_t)=\frac{\int_{\R^d}xe^{-\alpha\TE(x)}\mu_t(dx)}{\int_{\R^d}e^{-\alpha\TE(x)}\mu_t(dx)}\quad \mbox{ with }\mu_t=\operatorname{Law} (\OX_t)\,.
\end{equation}
Throughout this paper, we denote by $\mc{L}(X)$ the the law of random variable $X$. Applying the It\^{o}-Doeblin formula, one can see that $\mu$ is a weak solution to the following mean-field partial differential equation (PDE):
\begin{equation}\label{PDE}
\partial_t\mu_t=\frac{\sigma^2}{2}\sum_{k=1}^d\frac{\partial^2}{\partial {x_k}^2}( (x-X_\alpha(\mu_t))_k^2\mu_t)+\lambda\nabla
\cdot((x-X_\alpha(\mu_t))\mu_t),
\end{equation}
in the sense of Definition \ref{defweak}. We refer to  \cite[Theorem 2.1]{carrillo2018analytical} for the well-posedness of the particle system \eqref{particle} and \cite[Theorem 3.2]{carrillo2018analytical} for  the nonlinear SDE \eqref{Mckean}.  While the existence of the weak solution $\mu$ to PDE \eqref{PDE} follows straightforwardly from an application of the It\^{o}-Doeblin formula, the uniqueness may be obtained without much effort on the basis of the well-possedness of Mckean process \eqref{Mckean}; see Lemma \ref{lemuni} and the appendix for a sketched proof.

This paper is devoted to solving the open problem suggested in \cite{carrillo2018analytical,PTTM,totzeck2021trends}  by providing a rigorous proof of the mean-field limit for the CBO method \eqref{particle} through a tightness argument. We first prove that the sequence of empirical measures $\{\mu^N\}_{N\geq 2}$ ($\mu^N=\frac{1}{N}\sum_{i=1}^{N}\delta_{X^{i,N}}$ are $\mc{P}(\CR)$-valued random variables) is tight.  Prokhorov's theorem indicates that there exists a subsequence of $\{\mu^N\}_{N\geq 2}$ converging in law to a random measure $\mu$. Then, to identify the limit, we verify that the limit measure $\mu$ is a weak solution to the mean-field PDE \eqref{PDE} underlying the process \eqref{Mckean} almost surely, while the uniqueness of the weak solution to PDE \eqref{PDE} yields that $\mu$ is actually deterministic. The approach mixes certain probabilistic and stochastic arguments and some analysis on PDEs. For such a probabilistic method with tightness arguments, we refer to \cite{sznitman1991topics} for an introduction and the interested readers are also referred to \cite{liu2016propagation,li2019mean,fournier2014propagation} for the application to the study of the propagation of chaos for the large Brownian particle system with particular \textit{Coulomb type} interaction forces.

Throughout this paper the cost function $\TE$ satisfies the following assumption.
\begin{assum}\label{asum}
	For the given cost function $\TE:\RR^d\rightarrow \RR$, it holds that:
	\begin{itemize}
		\item[(1)]  There exists some constant $L>0$ such $|\TE(x)-\TE(y)|\leq L(|x|+|y|)|x-y|$ for all $x,y\in\RR^d$;
		\item[(2)]	$\TE$ is  bounded from below with $-\infty<\underline{\TE}:=\inf \TE$ and there exists some constant $C_u>0$ such that
		\begin{equation*}
      \TE(x)-\underline{\TE}\leq C_u(1+|x|^2)\mbox{ for all }x\in\RR^d\,;
		\end{equation*}
		\item[(3)]  $\TE$ has quadratic growth at infinity. Namely, there  exist constants $C_l,\, M>0$ such that
		\begin{equation*}
		\TE(x)-\underline{\TE}\geq C_l|x|^2\mbox{ for all }|x|\geq M\,.
		\end{equation*}
	\end{itemize}
\end{assum}

The rest of the paper is organized as follows: In Section 2 we prove the tightness of the empirical measures $\{\mu^N\}_{N\geq 2}$ associated to the CBO particle system \eqref{particle}  through the Aldous criteria; see Theorem \ref{thmtight}. Then in Section 3 we verify that the limit measure  $\mu$ of a subsequence of $\{\mu^N\}_{N\geq 2}$  is the unique weak solution to the mean-field CBO equation \eqref{PDE}; see Theorem \ref{thmmean}. In Section 4, the result is extended to the model of particle swarm optimization. Finally, the existence and uniqueness of the weak solution is proved for a class of linear PDEs in Appendix.

\section{Tightness of the empirical measures}
First, let us recall the following lemma on a uniform moment estimate for the particle system \eqref{particle} from \cite[Lemma 3.4]{carrillo2018analytical}
\begin{lem}\label{lem2mon}
	Let $\TE$ satisfy Assumption \ref{asum} and $\mu_0\in \mc{P}_4(\RR^d)$. For any $N\geq 2$, assume that $\{(X_t^{i,N})_{t\in[0,T]}\}_{i=1}^N$ is the unique solution to the particle system \eqref{particle} with $\mu_0^{\otimes N}$-distributed initial data $\{X_0^{i,N}\}_{i=1}^N$. Then there exists a constant $K>0$ independent of $N$ such that
	\begin{equation}
	\sup\limits_{i=1,\cdots,N} \left\{
		\sup\limits_{t\in[0,T]}
	\EE\left[|X_t^{i,N}|^2+|X_t^{i,N}|^4\right] +\sup\limits_{t\in[0,T]}\EE\left[|X_\alpha(\mu^N_t)|^2+|X_\alpha(\mu^N_t)|^4\right]
	\right\}
	\leq K\,. \label{est-lem2mon}
	\end{equation}
\end{lem}
We treat $X^{i,N}: \Omega\mapsto \CR$. Then
$\mu^N=\sum_{i=1}^{N}\delta_{X^{i,N}}: \Omega\mapsto \mc{P}(\CR)$ is a random measure. Let us denote $\mc{L}(\mu^N):=\mbox{Law}(\mu^N)\in \mc{P}(\mc{P}(\CR))$.  We can prove that $\{\mc{L}(\mu^N)\}_{N\geq 2}$ is tight, or  we say
$\{\mu^N\}_{N\geq1}$ is tight.

\begin{thm}\label{thmtight}
	Under the same assumption as in Lemma \ref{lem2mon}, recall the empirical measure $\mu^N=\frac{1}{N}\sum_{i=1}^N\delta_{X^{i,N}}$. Then
    the sequence $\{\mc{L}(\mu^N)\}_{N\geq 2}$ is tight in $\mc{P}(\mc{P}(\CR))$.
\end{thm}
\begin{proof}
	According to \cite[Proposition 2.2 (ii)]{sznitman1991topics}, we only need to prove that $\{\mc{L}(X^{1,N})\}_{N\geq 2}$ is tight in $\mc{P}(\CR)$ because of the exchangeability of the particle system. We shall do this by verifying the Aldous criteria below.
	\begin{lem}\label{lemAldous}
		Let $\{X^n\}_{n\in \NN}$ be a sequence of random variables defined on a probability space $(\Omega,\mc{F},\PP)$ and valued in $\mc{C}([0,T];\RR^d)$. The sequence of probability distributions $\{\mu_{X^n}\}_{n\in \NN}$  of $\{X^n\}_{n\in \NN}$ is tight on $\mc{C}([0,T];\RR^d)$ if the following two conditions hold.
		
		$(Con 1)$ For all $t\in [0,T]$, the set of distributions of $X_t^n$, denoted by $\{\mu_{X_t^n}\}_{n\in\NN}$, is tight as a sequence of probability measures on $\RR^d$.
		
		$(Con 2)$ For all $\varepsilon>0$, $\eta>0$, there exists $\delta_0>0$ and $n_0\in\NN$ such that for all $n\geq n_0$ and for all discrete-valued $\sigma(X^n_s;s\in[0,T])$-stopping times $\beta$ with $0\leq \beta+\delta_0\leq T$, it holds that
		\begin{equation}
		\sup_{\delta\in[0,\delta_0]}\PP\left(|X^n_{\beta+\delta}-X^n_{\beta}|\geq \eta\right)\leq \varepsilon\,.
		\end{equation}
	\end{lem}
It is now  sufficient to justify conditions $(Con 1)$ and $(Con 2)$:

$\bullet$ \textit{Step 1: Checking $(Con 1)$. }  
For any $\varepsilon>0$, there exists a compact subset $U_\varepsilon:=\{x:~|x|^2\leq \frac{K}{\varepsilon}\}$ such that by Markov's inequality
\begin{equation*}
\mc{L}(X_t^{1,N})~\big((U_\varepsilon)^c\big)=\PP\left(|X_t^{1,N}|^2> \frac{K}{\varepsilon}\right)\leq \frac{\varepsilon\EE[|X_t^{1,N}|^2]}{K}\leq\varepsilon,\quad \forall ~N\geq 2\,,
\end{equation*}
where we have used Lemma \ref{lem2mon} in the last inequality.
This means that for each $t\in[0,T]$, the sequence $\{\mc{L}(X_t^{1,N})\}_{N\geq 2}$ is tight, which verifies  condition $(Con 1)$ in Lemma \ref{lemAldous}.

$\bullet$ \textit{Step 2: Checking $(Con 2)$. }   Let $\beta$ be a $\sigma(X_s^{1,N};s\in[0,T])$-stopping time with discrete values such that $\beta+\delta_0\leq T$. 
Recalling \eqref{particle}, we have
\begin{align*}
X_{\beta+\delta}^{1,N}-X_{\beta}^{1,N}=-\int_{\beta}^{\beta+\delta }\lambda(X_s^{1,N}-X_\alpha(\mu_s^N))ds+\sigma\int_{\beta}^{\beta+\delta}D(X_s^{1,N}-X_\alpha(\mu_s^N))dB_s^1\,.
\end{align*}
Notice that
\begin{align}\label{es1}
&\EE\left[\left|\int_{\beta}^{\beta+\delta }\lambda(X_s^{1,N}-X_\alpha(\mu_s^N))ds\right|^2\right]\leq \lambda^2\delta \int_0^T\EE\left[|X_s^{1,N}-X_\alpha(\mu_s^N)|^2\right]ds\notag\\
\leq & 2 \lambda^2\delta T \left(\sup\limits_{t\in[0,T]}\EE\left[|X_t^{1,N}|^2\right] +\sup\limits_{t\in[0,T]}\EE\left[|X_\alpha(\mu^N_t)|^2\right]\right)\leq 2TK\lambda^2\delta\,,
\end{align}
where we have used Lemma \ref{lem2mon} in the last inequality. Further we apply It\^{o}'s isometry 
\begin{align}\label{es2}
&\EE\left[\left|\sigma\int_{\beta}^{\beta+\delta}D (X_s^{1,N}-X_\alpha(\mu_s^N))dB_s^1\right|^2\right]
= \sigma^2\EE\left[\int_{\beta}^{\beta+\delta}|X_s^{1,N}-X_\alpha(\mu_s^N)|^2ds\right]
		\notag\\
&\leq
	 \sigma^2 \delta^{\frac{1}{2}}  \EE \left[\left(\int_{0}^{T}|X_s^{1,N}-X_\alpha(\mu_s^N)|^4ds\right)^{\frac{1}{2}}\right]
	 \leq  \sigma^2 \delta^{\frac{1}{2}}
	  \left(\int_0^T\EE[|X_s^{1,N}-X_\alpha(\mu_s^N)|^4] ds\right)^\frac{1}{2}
	  	\notag\\
&\leq
	\sigma^2 \delta^{\frac{1}{2}} T^{\frac{1}{2}}(8K)^{\frac{1}{2}}\,.
\end{align} 
Combining estimates \ref{es1} and \ref{es2} one has
\begin{equation}
\EE[|X_{\beta+\delta}^{1,N}-X_{\beta}^{1,N}|^2]\leq C(\lambda,\sigma,T,K)\left( \delta^{\frac{1}{2}} + \delta\right)  \,.
\end{equation}
Hence, for any $\varepsilon>0$, $\eta>0$, there exists some $\delta_0>0$ such that for all $N\geq 2$ it holds that
\begin{equation}
\sup_{\delta\in[0,\delta_0]}\PP\left(|X_{\beta+\delta}^{1,N}-X_{\beta}^{1,N}|^2\geq \eta\right)\leq \sup_{\delta\in[0,\delta_0]}\frac{\EE\left[|X_{\beta+\delta}^{1,N}-X_{\beta}^{1,N}|^2\right]}{\eta}\leq \varepsilon\,.
\end{equation}
This justifies condition $Con 2$ in Lemma \ref{lemAldous} and completes the proof of Theorem \ref{thmtight}.  
	\end{proof}

As a consequence of the tightness in Theorem \ref{thmtight}, we obtain the following results.
\begin{lem}\label{lemtight}
	\begin{enumerate}
		\item  There exist a subsequence of $\{\mu^N\}_{N\geq 2}$ (denoted w.l.o.g. by itself) and a random measure $\mu: \Omega \mapsto \mc{P}(\CR)$ such that
		\begin{equation}
\mu^N\rightharpoonup\mu \mbox{ in law as }N\to \infty \,,
		\end{equation}
		which is equivalently  to say $\mc{L}(\mu^N)$ converges weakly to $\mc{L}(\mu)$ in $\mc{P}(\mc{P}(\CR))$
		\item For the subsequence in $(1)$, the time marginal $\mu_t^N$ of  $\mu^N$,  as $\mc{P}(\RR^d)$ valued random measure converges in law to $\mu_t\in \mc{P}(\RR^d)$, the time marginal of $\mu$. Namely $\mc{L}(\mu_t^N)$ converges weakly to $\mc{L}(\mu_t)$ in $\mc{P}(\mc{P}(\R^d))$
	\end{enumerate}
\end{lem}
\begin{proof}
	By Prokhorov's theorem, assertion (1) follows form the tightness of $\{\mc{L}(\mu^N)\}_{N\geq 2}$ in $\mc{P}(\mc{P}(\CR))$ as shown in Theorem \ref{thmtight}. 
	
	As for assertion $(2)$, we first notice that a sequence $\nu^n \in \mc{P}(\CR) $ that converges weakly to $\nu \in \mc{P}(\CR)$ will imply that $\nu_t^n \in \mc{P}(\R^d) $ converges weakly to $\nu_t\in\mc{P}(\R^d)$ for each time $t$. Indeed, for each $\phi\in\mc{C}_b(\R^d)$, we have  $ \int_{\R^d}\phi(x)\nu_t^n(dx) = \int_{\CR}\phi(\textbf{x}_t)\nu^n(d\textbf{x})  $ (see \cite[Lemma 2.8]{li2019mean}). Note that for all $\textbf{x}\in \CR$, $\textbf{x}\mapsto \phi(\textbf{x}_t)$ is a bounded continuous functional on $\CR$, which leads to 
	\begin{align*}
	\lim\limits_{n\to \infty }   \int_{\R^d}\phi(x)\nu_t^n(dx) 
=\lim\limits_{n\to \infty }   \int_{\CR}\phi(\textbf{x}_t)\nu^n(d\textbf{x})  
	&=  \int_{\CR}\phi(\textbf{x}_t)\nu(d\textbf{x})  
	&
	=\int_{\R^d}\phi(x)\nu_t(dx) \,.
	\end{align*}
	
	Now we consider a bounded continuous functional $\Gamma : \mc{P}(\R^d)\to \RR$, then one defines $\Gamma_1: \mc{P}(\CR)\to \RR$ as
	\begin{equation}
	\Gamma_1 (\nu):=\Gamma(\nu_t)\mbox{ for any }\nu \in \mc{P}(\CR)\,.
	\end{equation}
	This means that $\Gamma_1$ is a bounded continuous functional on $\mc{P}(\CR)$ according to what has been justified. Consequently, 
	\begin{equation}
	\EE[\Gamma_1(\mu^N)]\to \EE[\Gamma_1(\mu)]\Rightarrow \EE[\Gamma(\mu_t^N)]\to \EE[\Gamma(\mu_t)]\,,
	\end{equation}
	which implies assertion $(2)$.
	\end{proof}

\section{Identification of the limit measure via PDE \eqref{PDE}}

\begin{definition}\label{defweak}
	We say $\mu_t\in \mc{C}([0,T];\mc{P}_2(\RR^d))$ is a weak solution to PDE \eqref{PDE}	 if 
\begin{itemize}
	\item[(i)] The continuity in time is in $\mc{C}_b'$ topology, namely it holds
	 \begin{equation}\label{continu}
\int_{\R^d} \phi(x)\mu_{t_n}(dx)\to \int_{\R^d} \phi(x)\mu_{t}(dx)
	 \end{equation}
	for all $\phi\in \mc{C}_b(\R^d)$ and $t_n\to t$\,;
	\item[(ii)] The following holds
	\begin{align}
	&\la \varphi(x) ,\mu_t (dx)\ra-\la \varphi(x),\mu_0(dx) \ra +\lambda\int_0^t\left\la (x-X_\alpha(\mu_s))\cdot \nabla\varphi(x), \mu_s (dx)\right\ra ds\notag\\
	&\quad-\frac{\sigma^2}{2}\int_0^t \sum_{k=1}^d \left\la (x-X_\alpha(\mu_s))_k^2\frac{\partial^2}{\partial{x_k}^2}\varphi(x),\mu_s (dx)\right\ra ds=0, 
	\end{align}
	for all $\varphi\in \mc{C}_c^2(\RR^d)$\,.
\end{itemize}
\end{definition}

First, for each $\varphi\in \mc{C}_c^2(\R^d)$, we define a functional on $\mc{P}(\CR)$ as following
\begin{align}
F_{\varphi}(\nu)&:=\la \varphi(\textbf{x}_t),\nu(d\textbf{x})\ra-\la \varphi(\textbf{x}_0),\nu(d\textbf{x}) \ra +\lambda\int_0^t\left\la (\textbf{x}_s-X_\alpha(\nu_s))\cdot \nabla\varphi(\textbf{x}_s), \nu(d\textbf{x})\right\ra ds\notag\\
&\quad-\frac{\sigma^2}{2}\int_0^t \sum_{k=1}^d \left\la (\textbf{x}_s-X_\alpha(\nu_s))_k^2\frac{\partial^2}{\partial{x_k}^2}\varphi(\textbf{x}_s),\nu (d\textbf{x})\right\ra ds\notag\\
&=\la \varphi(x) ,\nu_t (dx)\ra-\la \varphi(x),\nu_0(dx) \ra +\lambda\int_0^t\left\la (x-X_\alpha(\nu_s))\cdot \nabla\varphi(x), \nu_s (dx)\right\ra ds\notag\\
&\quad-\frac{\sigma^2}{2}\int_0^t\sum_{k=1}^d\left\la (x-X_\alpha(\nu_s))_k^2\frac{\partial^2}{\partial{x_k}^2}\varphi(x),\nu_s (dx)\right\ra ds \,,
\end{align}
for all $\nu\in \mc{P}(\CR)$ and $\textbf{x}\in \CR$. Recall that here
\begin{equation}
X_\alpha(\nu_s) =\frac{\int_{\R^d}xe^{-\alpha\TE(x)}\nu_s(dx)}{\int_{\R^d}e^{-\alpha\TE(x)}\nu_s(dx)}=:\frac{\la xe^{-\alpha\TE(x)},\nu_s(dx)\ra}{\la e^{-\alpha\TE(x)}, \nu_s(dx)\ra}\,.
\end{equation}

Then we have the following estimate.
\begin{proposition}\label{prop}
	Let $\TE$ satisfy Assumption \ref{asum} and $\mu_0\in \mc{P}_4(\RR^d)$. For any $N\geq 2$, assume that $\{(X_t^{i,N})_{t\in[0,T]}\}_{i=1}^N$ is the unique solution to the particle system \eqref{particle} with $\mu_0^{\otimes N}$-distributed initial data $\{X_0^{i,N}\}_{i=1}^N$. There exists a constant $C>0$ depending only on $\sigma,K,T$, and $\norm{\nabla\varphi}_\infty$ such that
\begin{equation}
\EE[|F_{\varphi}(\mu^N)|^2]\leq \frac{C}{N}\,,
\end{equation}
where $\mu^N=\frac{1}{N}\sum_{i=1}^N\delta_{X^{i,N}}$ is the empirical measure.
\end{proposition}
\begin{proof}
	Using the definition of $F_{\varphi}$ one has
	\begin{align}
	F_{\varphi}(\mu^N)
	&=\frac{1}{N}\sum_{i=1}^{N}\varphi(X_t^{i,N})-\frac{1}{N}\sum_{i=1}^{N}\varphi(X_0^{i,N}) +\lambda\int_0^t \frac{1}{N}\sum_{i=1}^{N}(X_s^{i,N}-X_\alpha(\mu_s^N))\cdot\nabla \varphi(X_s^{i,N})ds\notag\\
	&\quad  -\frac{\sigma^2}{2}\int_0^t\frac{1}{N}\sum_{i=1}^{N} \sum_{k=1}^d (X_s^{i,N}-X_\alpha(\mu_s^N))_k^2\frac{\partial^2}{\partial{x_k}^2}\varphi(X_s^{i,N}) ds\notag \\
	&=\frac{1}{N}\sum_{i=1}^{N}\bigg(\varphi(X_t^{i,N})-\varphi(X_0^{i,N}) +\lambda\int_0^t (X_s^{i,N}-X_\alpha(\mu_s^N))\cdot \nabla\varphi(X_s^{i,N})ds\notag\\
	&\qquad-\frac{\sigma^2}{2}\int_0^t\sum_{k=1}^d (X_s^{i,N}-X_\alpha(\mu_s^N))_k^2\frac{\partial^2}{\partial{x_k}^2}\varphi(X_s^{i,N}) ds\bigg)
	\,.
	\end{align}
	For each $i=1,\cdot,N$ and $\varphi \in \mc{C}_c^2(\R^d)$, applying It\^{o}-Doeblin formula gives
	\begin{align}
	\varphi(X_t^{i,N})&=	\varphi(X_0^{i,N})-\int_0^t\lambda\nabla\varphi(X_s^{i,N})\cdot (X_s^{i,N}-X_\alpha(\mu_s^N))ds
	+\sigma\int_0^t D(X_s^{i,N}-X_\alpha(\mu_s^N))\nabla\varphi(X_s^{i,N})\cdot dB_s^i\notag\\
	&\quad +\frac{\sigma^2}{2}\int_0^t \sum_{k=1}^d (X_s^{i,N}-X_\alpha(\mu_s^N))_k\frac{\partial^2}{\partial{x_k}^2}\varphi(X_s^{i,N}) ds\,.
	\end{align}
	This implies that
	\begin{equation}
	F_{\varphi}(\mu^N)=\frac{\sigma}{N}\sum_{i=1}^{N}\int_0^t D(X_s^{i,N}-X_\alpha(\mu_s^N))\nabla\varphi(X_s^{i,N})\cdot dB_s^i\,.
	\end{equation}
	Then it holds that
	\begin{align}
	\EE[|F_{\varphi}(\mu^N)|^2]&=\frac{\sigma^2}{N^2}\EE\left[\left|\sum_{i=1}^{N}\int_0^t D(X_s^{i,N}-X_\alpha(\mu_s^N))\nabla\varphi(X_s^{i,N})\cdot dB_s^i\right|^2\right]\notag\\
	&=\frac{\sigma^2}{N^2}\sum_{i=1}^{N}\EE\left[\left|\int_0^t D(X_s^{i,N}-X_\alpha(\mu_s^N))\nabla\varphi(X_s^{i,N})\cdot dB_s^i\right|^2\right]\notag \\
		&=\frac{\sigma^2}{N^2}\sum_{i=1}^{N}\EE\left[\int_0^t |X_s^{i,N}-X_\alpha(\mu_s^N)|^2|\nabla\varphi(X_s^{i,N})|^2ds\right]\notag\\
		&\leq C(\sigma,K,T,\norm{\nabla\varphi}_\infty)\frac{1}{N}\,,
	\end{align}
	where we have used Lemma \ref{lem2mon} in the last inequality. This completes the proof.
	\end{proof}
	


 By Skorokhod's lemma (see \cite[Theorem 6.7 on page
70]{billingsley1999convergence}),   using Lemma \ref{lemtight} we may find a common probability space $(\Omega,\mc{F},\PP)$ on which the processes $\{\mu^N\}_{N\in\mathbb N}$ converge to some process $\mu$ as a random variable valued in $\mc{P}(\CR)$ almost surely. In particular, we have that for all $t\in [0,T]$ and $\phi\in C_b(\RR^d)$,
\begin{equation}\label{310}
\lim_{N\rightarrow \infty} |\la \phi,\mu_t^N-\mu_t\ra| + \left|X_\alpha(\mu^N_t)-X_\alpha(\mu_t)\right|= 0,\quad \text{a.s.}
\end{equation}
Indeed, according to Assumption \ref{asum}, one has  $xe^{-\alpha \TE(x)}, e^{-\alpha \TE(x)} \in \mc{C}_b(
\R^d)$, which gives
\begin{align}
	\lim_{N\rightarrow \infty} X_\alpha(\mu_t^N) =	\lim_{N\rightarrow \infty} \frac{\la xe^{-\alpha\TE(x)},\mu_t^N(dx)\ra}{\la e^{-\alpha\TE(x)}, \mu_t^N(dx)\ra}=\frac{\la xe^{-\alpha\TE(x)},\mu_t(dx)\ra}{\la e^{-\alpha\TE(x)}, \mu_t(dx)\ra}=X_\alpha(\mu_t)\quad \text{a.s.} 
\end{align}
\begin{lem}{\cite[Lemma 3.3]{carrillo2018analytical}}\label{lemXa}
	Let $\TE$ satisfy Assumption \ref{asum}  and $\mu\in \mc{P}_2(\R^d)$. Then  it holds that
	\begin{equation}
	|X_\alpha(\mu)|^2 \leq b_1+b_2\int_{\RR^d}|x|^2\mu(dx)\,,
	\end{equation}
	where $b_1$ and $b_2$ depends only on $M$, $C_u$, and $C_l$.
\end{lem}

For each $A>0$, let us take $\phi=|\cdot|^4\wedge A \in\mc{C}_b(\R^d)$. It follows from \eqref{310} that
\begin{align}
\EE\left[\int_{\R^d}(|x|^4\wedge A)\mu_t(x)\right]=\EE\left[\lim_{N\rightarrow \infty}\int_{\R^d}(|x|^4\wedge A)\mu_t^N(x)\right]\leq \lim_{N\rightarrow \infty}\frac{\sum_{i=1}^{N}\EE[|X_t^{i,N}|^4]}{N}\leq K\,,
\end{align}
where we have used Lemma \ref{lem2mon}. Letting $A\rightarrow \infty$, we have 
\begin{align}
\sup_{t\in[0,T]}\EE\left[\int_{\R^d}|x|^4\mu_t(x)\right]\leq K. \label{bd-mu}
\end{align}
 Then Lemma \ref{lemXa} implies that
\begin{equation}\label{Xa1}
\EE[|X_\alpha(\mu_t)|^4]<\infty\,,
\end{equation}
for all $t\in[0,T]$.

Furthermore, it holds that 
\begin{equation}\label{convergence}
\lim_{N\rightarrow \infty} 
\EE\left[
\left|\la \phi,\mu_t^N-\mu_t\ra \right|^2 + |X_\alpha(\mu^N_t)-X_\alpha(\mu_t)|^2
\right]=0,
\end{equation}
 which follows directly from the pointwise convergences of $\la \phi,\mu_t^N-\mu_t\ra$ and $X_\alpha(\mu^N_t)-X_\alpha(\mu_t)$, and the uniform estimate \eqref{est-lem2mon} in Lemma \ref{lem2mon} and \eqref{Xa1}. To see this, let us consider a sequence of random variables $\{X_n\}_{n\geq 1}$, which satisfies that $X_n\to 0~~(n\to\infty)$ pointwisely and $\sup\limits_{n\geq 1 }\EE[|X_n|^4]\leq C$ uniformly in $n$. For all $A>0$, we compute
 \begin{align}
 \EE[|X_n|^2]&=\EE\left[|X_n|^2\mathbb{I}_{|X_n|\leq A}\right]+\EE\left[|X_n|^2\mathbb{I}_{|X_n|> A}\right]\notag\\
 &\leq  \EE\left[|X_n|^2\mathbb{I}_{|X_n|\leq A}\right]+(\EE[|X_n|^4])^{\frac{1}{2}}(\EE\left[\mathbb{I}_{|X_n|> A}\right])^{\frac{1}{2}}
 \,.
 \end{align}
It is obvious that $\EE\left[|X_n|^2\mathbb{I}_{|X_n|\leq A}\right]\to 0 ~~(n\to\infty)$ holds by the dominated convergence theorem. One also notices that
\begin{equation}
(\EE[|X_n|^4])^{\frac{1}{2}}(\EE\left[\mathbb{I}_{|X_n|> A}\right])^{\frac{1}{2}}\leq (\EE[|X_n|^4])^{\frac{1}{2}}\frac{(\EE[|X_n|^4])^{\frac{1}{2}}}{A^2}\leq \frac{C}{A^2}\to 0\quad \mbox{ as }A\to\infty\,,
\end{equation}
which leads to $\EE\left[|X_n|^2\right]\to 0$ as $n\to\infty$.


\begin{thm}\label{thmmean}
	Let $\TE$ satisfy Assumption \ref{asum} and $\mu_0\in \mc{P}_4(\RR^d)$. For any $N\geq 2$, assume that $\{(X_t^{i,N})_{t\in[0,T]}\}_{i=1}^N$ is the unique solution to the particle system \eqref{particle} with $\mu_0^{\otimes N}$-distributed initial data $\{X_0^{i,N}\}_{i=1}^N$.   Then the limit (denoted by $\mu$) of the sequence of the empirical measure $\mu^N=\frac{1}{N}\sum_{i=1}^N\delta_{X^{i,N}}$ exists. Moreover, $\mu$ is deterministic and it is the unique weak solution to PDE \eqref{PDE}.
\end{thm}
\begin{proof}
Suppose the $\mc{P}(\CR)$-valued random variable $\mu$ is the limit of a subsequence of the empirical measure $\mu^N=\frac{1}{N}\sum_{i=1}^N\delta_{X^{i,N}}$. W.l.o.g., Denote the subsequence by itself. We may continue to work on the above common probability space $(\Omega,\mc{F},\PP)$ by Skorokhod's lemma where the convergence is holding almost surely (see \eqref{310} for instance). We may first check that  $\mu_t$ is a.s. continuous in time in the sense of \eqref{continu}. Indeed for any $\phi\in\mc{C}_b(\R^d)$ and $t_n\to t$ we may apply dominated convergence theorem
\begin{equation*}
\int_{\CR}\phi(\textbf{x}_{t_n})\mu(d\textbf{x})\to \int_{\CR}\phi(\textbf{x}_{t})\mu(d\textbf{x})\quad \text{a.s.,}
\end{equation*}
which gives
\begin{equation*}
\int_{\R^d}\phi(x)\mu_{t_n}(d{x})\to \int_{\R^d}\phi(x)\mu_t(d{x})\quad\text{a.s.}
\end{equation*}

For $\varphi\in \mc{C}_c^2(\RR^d)$, using the convergence result in \eqref{convergence} one has
	\begin{equation}\label{est1}
	\lim_{N\rightarrow \infty} \EE\left[|(\la \varphi(x) ,\mu_t^N (dx)\ra-\la \varphi(x),\mu_0^N(dx) \ra) -(\la \varphi(x) ,\mu_t (dx)\ra-\la \varphi(x),\mu_0(dx) \ra)|\right]=0\,.
	\end{equation}
	
     Further we notice that
{\small      \begin{align}
     &\left|\int_0^t\la (x-X_\alpha(\mu_s^N))\cdot \nabla\varphi(x), \mu_s^N (dx)\ra ds -\int_0^t\la (x-X_\alpha(\mu_s))\cdot \nabla\varphi(x), \mu_s (dx)\ra ds\right|\notag\\
     \leq& \int_0^t\left|\la (x-X_\alpha(\mu_s^N))\cdot \nabla\varphi(x), \mu_s^N (dx)-\mu_s (dx)\ra \right| ds +\int_0^t\left|\la (X_\alpha(\mu_s)-X_\alpha(\mu_s^N))\cdot \nabla\varphi(x), \mu_s (dx)\ra\right| ds \notag\\
     =:&\int_0^t|I_1^N(s)|ds+\int_0^t|I_2^N(s)|ds\,.
     \end{align}
 }
One computes
\begin{align}
\EE[|I_1^N(s)|]&\leq \EE[|\la x\cdot \nabla \varphi(x),\mu_s^N (dx)-\mu_s (dx)\ra|]+\EE[|X_\alpha(\mu_s^N) \cdot \la\nabla \varphi(x),\mu_s^N (dx)-\mu_s (dx)\ra|]\notag\\
&\leq \EE[|\la x\cdot \nabla \varphi(x),\mu_s^N (dx)-\mu_s (dx)\ra|]+K^{\frac{1}{2}}(\EE[| \la\nabla \varphi(x),\mu_s^N (dx)-\mu_s (dx)\ra|^2])^\frac{1}{2}\,,
\end{align}
where we have used Lemma \ref{lem2mon} in the second inequality.
Since $\varphi$ has a compact support, applying \eqref{convergence} leads to 
\begin{equation}
\lim\limits_{N\to \infty }\EE\left[|I_1^N(s)|\right]=0\,.
\end{equation}
Moreover, the uniform boundedness of $\EE\left[|I_1^N(s)|\right]$ follows directly from \eqref{bd-mu}, \eqref{Xa1}, and the estimates in Lemma \ref{lem2mon}, which by the dominated convergence theorem implies
\begin{equation}\label{I1}
\lim\limits_{N\to \infty }\int _0^t\EE[|I_1^N(s)|]ds=0\,.
\end{equation}
	
	As for  $I_2^N$, we know that
	\begin{equation}
	 \left|\la (X_\alpha(\mu_s)-X_\alpha(\mu_s^N))\cdot \nabla\varphi(x), \mu_s (dx)\ra\right|\leq \norm{\nabla \varphi}_\infty  |X_\alpha(\mu_s)-X_\alpha(\mu_s^N)|\,.
	\end{equation}
	 Hence by \eqref{convergence} it yields that
	\begin{equation}
	\lim\limits_{N\to \infty }\EE[|I_2^N(s)|]=0\,.
	\end{equation}
	Again by the dominated convergence theorem, we have
	\begin{equation}
	\lim\limits_{N\to \infty }\int _0^t\EE[|I_2^N(s)|]ds=0\,.
	\end{equation}
	This combined with \eqref{I1} leads to
	\begin{equation}\label{est2}
\lim_{N\rightarrow \infty}	\EE\left[\left|\int_0^t\la (x-X_\alpha(\mu_s^N))\cdot \nabla\varphi(x), \mu_s^N (dx)\ra ds -\int_0^t\la (x-X_\alpha(\mu_s))\cdot \nabla\varphi(x), \mu_s (dx)\ra ds\right|\right]=0\,.
	\end{equation}

Similarly we split the error
{\small \begin{align}
&\left|\int_0^t\la (x-X_\alpha(\mu_s^N))_k^2\frac{\partial^2}{\partial{x_k}^2}\varphi(x),\mu_s^N (dx)\ra ds -\int_0^t\la (x-X_\alpha(\mu_s))_k^2\frac{\partial^2}{\partial{x_k}^2}\varphi(x),\mu_s (dx)\ra ds \right|\notag\\
\leq &\left|\int_0^t\la (x-X_\alpha(\mu_s^N))_k^2\frac{\partial^2}{\partial{x_k}^2}\varphi(x),\mu_s^N (dx)-\mu_s(dx)\ra ds\right|
+\left|\int_0^t\la ((x-X_\alpha(\mu_s^N))_k^2-(x-X_\alpha(\mu_s))_k^2)\frac{\partial^2}{\partial{x_k}^2}\varphi(x),\mu_s (dx)\ra ds \right| \notag\\
=:& \int_0^t |I_3^N(s)|ds+\int_0^t |I_4^N(s)|ds\,.
\end{align}}
Following the same argument as for $I_1^N$ and $I_2^N$, one has
	\begin{equation}
\lim\limits_{N\to \infty }\int_0^t \EE[|I_3^N(s)|]ds=0 \mbox{ and } \lim\limits_{N\to \infty }\int_0^t \EE[|I_4^N(s)|]ds=0\,.
\end{equation}
	This implies that
{\small 	\begin{equation}\label{est3}
	\lim_{N\rightarrow \infty}\EE\left[\left|\int_0^t \sum_{k=1}^d\la (x-X_\alpha(\mu_s^N))_k^2\frac{\partial^2}{\partial{x_k}^2}\varphi(x),\mu_s^N (dx)\ra ds -\int_0^t\sum_{k=1}^d\la (x-X_\alpha(\mu_s))_k^2\frac{\partial^2}{\partial{x_k}^2}\varphi(x),\mu_s (dx)\ra ds\right|\right]=0\,.
	\end{equation}}

Collecting estimates \eqref{est1}, \eqref{est2} and \eqref{est3} we have
\begin{equation}
\lim_{N\rightarrow \infty}\EE[|F_{\varphi}(\mu^N) -F_{\varphi}(\mu)|]=0 \,.
\end{equation}
Then we have
\begin{equation}
\EE[|F_{\varphi}(\mu)|]\leq \EE[|F_{\varphi}(\mu^N) -F_{\varphi}(\mu)|]+\EE[|F_{\varphi}(\mu^N)|]\leq \EE[|F_{\varphi}(\mu^N) -F_{\varphi}(\mu)|]+\frac{C}{\sqrt{N}}\to 0\quad \mbox{as }N\to\infty\,,
\end{equation}
where we have used Proposition \ref{prop} in the last inequality.
This implies that
\begin{equation}
F_{\varphi}(\mu)=0\quad \text{a.s.}
\end{equation}
In other words, it holds that
\begin{align}
&\la \varphi(x) ,\mu_t (dx)\ra-\la \varphi(x),\mu_0(dx) \ra +\lambda\int_0^t\la (x-X_\alpha(\mu_s))\cdot \nabla\varphi(x), \mu_s (dx)\ra ds\notag\\
&\quad-\frac{\sigma^2}{2}\int_0^t\sum_{k=1}^d\la (x-X_\alpha(\mu_s))_k^2\frac{\partial^2}{\partial{x_k}^2}\varphi(x),\mu_s (dx)\ra ds=0 \quad \text{a.s.}\,,
\end{align}
for any $\varphi\in \mc{C}_c^2(\RR^d)$. 

Until now we have proved that $\mu$ a.s. is a weak solution to PDE \eqref{PDE}. Finally combining the uniqueness of weak solution to \eqref{PDE} (see in Lemma \ref{lemuni} below) and the arbitrariness of the subsequence of $\{\mu^N\}_{N\geq 1}$,  the (deterministic) weak solution $\mu$ to PDE \eqref{PDE} must be the limit of the whole sequence $\{\mu^N\}_{N\geq 1}$. We complete the proof.	\end{proof}

\begin{lem}\label{lemuni}
	Assume that $\mu^1,\mu^2\in \mc{C}([0,T];\mc{P}_2(\RR^d))$ are two weak solutions to PDE \eqref{PDE}  in the sense of Definition \ref{defweak} with the same initial data $\mu_0$. Then it holds that 
	$$\sup\limits_{t\in[0,T]}W_2(\mu_t^1,\mu_t^2)=0\,,$$
	where $W_2$ is the 2-Wasserstein distance.
\end{lem}
\begin{proof}
	 We construct two linear processes $(\widehat X_t^i)_{t\in[0,T] }$ $(i=1,2)$ satisfying
	\begin{equation}
	d\widehat X_t^i=-\lambda(\widehat X_t^i-X_\alpha(\mu_t^i))dt+\sigma D(\widehat X_t^i-X_\alpha(\mu_t^i))dB_t\,, \label{Eqn for Vbar}
	\end{equation}
	with the common initial data $\widehat X_0$ distributed  according to $\mu_0$. Above processes are linear because that $\mu^i$ are prescribed.
	Let us denote $\rm{law}(\widehat X_t^i)=\hat {\mu}_t^i$ $(i=1,2)$, which are weak solutions to the following linear PDE
	\begin{align*}
	\partial_t\hat  \mu_t^i=\frac{\sigma^2}{2}\sum_{k=1}^d\frac{\partial^2}{\partial{x_k}^2}((x-X_\alpha(\mu_t^i))_k^2\hat \mu_t^i)-\lambda\nabla
	\cdot((x-X_\alpha(\mu_t^i))\hat \mu_t^i)\,,
	\end{align*}
	where $X_\alpha(\mu_t^i)\in \mc{C}([0,T];\R^d)$ for given $\mu^i\in \mc{C}([0,T],\mc{P}_2(\R^d))$.
	By  the uniqueness of weak solution to the above linear PDE (see Theorem \ref{thmapp} in Appendix) and the fact that $\mu^i$ is also a solution to the above PDE, it follows that
	$\hat  \mu_t^i=\mu_t^i$ $(i=1,2)$. Consequently, the process $(\widehat X_t^i)_{(t\in[0,T] )}=(\OX_t^i)_{(t\in[0,T] )}$ are solutions to the nonlinear SDE \eqref{Mckean}, for which the uniqueness has been obtained in \cite[Theorem 3.2]{carrillo2018analytical}. In particular, it holds that 
	\begin{equation}
	\sup\limits_{t\in[0,T]}\EE\left[|\OX_t^1-\OX_t^2|^2\right]=0\,,
	\end{equation}
	which by the definition of Wasserstein distance implies $$\sup\limits_{t\in[0,T]}W_2(\mu_t^1,\mu_t^2)=\sup\limits_{t\in[0,T]}W_2(\hat \mu_t^1,\hat  \mu_t^2)\leq\sup\limits_{t\in[0,T]}\EE[|\widehat X_t^1-\widehat X_t^2|^2] =\sup\limits_{t\in[0,T]}\EE[|\OX_t^1-\OX_t^2|^2] =0\,.$$ Thus the uniqueness is obtained.
	\end{proof}

\section{Mean-field limit for Particle Swarm Optimization}
In this section we extend our discussions to the model of particle swarm optimization (PSO) proposed recently by Grassi and Pareschi \cite{grassi2020particle}, where they only numerically verified the mean-limit result. We consider PSO based on a continuous description in the form of a system of stochastic differential equations:
\begin{align}\label{PSO}
\begin{cases}
dX_t^{i,N}= V_t^{i,N}d t, \\
dV_t^{i,N}=-\frac{\gamma}{m}V_t^{i,N}dt+\frac{\lambda}{m}(X^\alpha(\rho_t^{N})-X_t^{i,N})dt+\frac{\sigma}{m}D(X^\alpha(\rho_t^{N})-X_t^{i,N})dB_t^i,\quad i=1,\cdots,N\,,
\end{cases}
\end{align}
where the $\RR^d$-valued functions $X_t^{i,N} $ and $V_t^{i,N}$ denote the position and velocity of the $i$-th particle at time $t$, $m>0$ is the inertia weight, $\gamma=1-m\geq0$ is the friction coefficient,  $\lambda>0$ is the acceleration coefficient, $\sigma>0$ is the diffusion coefficient,  and $\{(B_t^i)_{t\geq0}\}_{i=1}^N$ are $N$ independent $d$-dimensional Brownian motions. 
Here the weighted average is given by
\begin{equation}\label{regularizer}
{X}^{\alpha}(\rho_t^{N}):=\frac{\int_{\RR^d}x\omega_{\alpha}^{\mc{E}}(x)\rho_t^{N}(dx)}{\int_{\RR^d}\omega_{\alpha}^{\mc{E}}(x)\rho_t^{N}(dx)},
\end{equation}
with  the empirical measure $\rho^N:=\frac{1}{N}\sum_{i=1}^{N}\delta_{X^{i,N}}$, which is the spacial marginal of
$$f^N=\frac{1}{N}\sum_{i=1}^{N}\delta_{(X^{i,N},V^{i,N})} : \Omega\mapsto \mc{P}(\CR\times \CR)\,.$$

Parallel to Theorem \ref{thmtight}, we can prove the tightness of the empirical measures $\{f^N\}_{N\geq2}$ by verifying the Aldous criteria presented in Lemma \ref{lemAldous}.
\begin{thm}
	Let $\TE$ satisfy Assumption \ref{asum} and $f_0\in \mc{P}_4(\RR^d\times\R^d)$. For any $N\geq 2$, we assume that $\{(X_t^{i,N},V_t^{i,N})_{t\in[0,T]}\}_{i=1}^N$ is the unique solution to the particle system \eqref{PSO} with $f_0^{\otimes N}$-distributed initial data $\{X_0^{i,N},V_0^{i,N}\}_{i=1}^N$. Then the sequence $\{\mc{L}(f^N)\}_{N\geq 2}$ is tight in $\mc{P}(\mc{P}(\CR\times\CR))$.
\end{thm}
\begin{proof}
	It is sufficient to justify conditions $(Con 1)$ and $(Con 2)$ in Lemma \ref{lemAldous}.
	
	$\bullet$ \textit{Step 1: Checking $(Con 1)$. }   It is obvious that
	\begin{equation}
	\EE[|X_t^{i,N}|^4]\leq 2^3\EE[|X_0^{i,N}|^4]+2^3T^3\int_0^t \EE[|V_s^{i,N}|^4]ds\,,
	\end{equation}
	holds for each $i=1,\cdots,N$. Applying Doob's martingale inequality,  we further obtain
	\begin{align}
	\EE[|V_t^{i,N}|^4]&\leq  4^3 \EE[|V_0^{i,N}|^4]+4^3\frac{\gamma^4}{m^4}T^3\int_0^t \EE[|V_s^{i,N}|^4]ds+4^3\frac{\lambda^4}{m^4}T^3\int_0^t \EE[|X^\alpha(\rho_s^{N})-X_s^{i,N}|^4]ds\nn\\
	&\quad +\frac{4^7\sigma^4}{3^4m^4}T\int_0^t \EE[|X^\alpha(\rho_s^{N})-X_s^{i,N}|^4]ds\nn\\
	&\leq C\EE[|V_0^{i,N}|^4]+C\int_0^t \EE[|V_s^{i,N}|^4+|X_s^{i,N}|^4]ds+C\int_0^t \EE[|X^\alpha(\rho_s^{N})|^4]ds\,,
	\end{align}
	where $C$ is independent of $N$. Thus it holds that
	\begin{equation}\label{total}
	\EE[|X_t^{i,N}|^4+|V_t^{i,N}|^4]\leq C\EE[|X_0^{i,N}|^4+|V_0^{i,N}|^4]+C\int_0^t \EE[|V_s^{i,N}|^4+|X_s^{i,N}|^4]ds+C\int_0^t \EE[|X^\alpha(\rho_s^{N})|^4]ds\,.
	\end{equation}
	Summing the above estimate over $i=1,\cdots,N$, dividing by $N$ and using the linearity of the expectation, we have
	\begin{align}\label{Gron}
	\EE\int (|x|^4+|v|^4)f_t^N(dx,dv)&\leq C\EE\int (|x|^4+|v|^4)f_0^N(dx,dv)+C\int_0^t (\EE\int (|x|^4+|v|^4)f_s^N(dx,dv))ds\nn\\
	&\quad +C\int_0^t \EE[|X^\alpha(\rho_s^{N})|^4]ds\,.
	\end{align}
	It follows from Lemma \ref{lemXa} that
	\begin{equation}
	|X^\alpha(\rho_s^{N})|^4\leq (b_1+b_2\int |x|^2 \rho_s^N(dx))^2\leq 2(b_1^2+b_2^2\int|x|^4 \rho_s^N(dx))\leq 2(b_1^2+b_2^2\int(|x|^4+|v|^4) f_s^N(dx,dv))\,.
	\end{equation}
	Inserting this into \eqref{Gron} and applying Gronwall's inequality yield that
	\begin{equation}
	\sup\limits_{t\in[0,T]}\EE\int (|x|^4+|v|^4)f_t^N(dx,dv)\leq K\,,
	\end{equation}
	where $K$ is independent of $N$. This implies $\sup\limits_{t\in[0,T]}\EE[|X^\alpha(\rho_s^{N})|^4]\leq 2(b_1^2+b_2^2K)$. Then applying Gronwall's inequality on \eqref{total} we have
	\begin{equation}\label{unibound}
	\sup\limits_{t\in[0,T]}\sup\limits_{i=1,\cdots,N}\EE[|X_t^{i,N}|^4+|V_t^{i,N}|^4]\leq K'\,,
	\end{equation}
	where $K'>0$ is independent of $N$. Then $(Con 1)$ may be verified in a similar way to Theorem \ref{thmtight}.
	
	$\bullet$ \textit{Step 1: Checking $(Con 2)$. }   Let $\beta$ be a $\sigma((X_s^{1,N},V_s^{1,N});s\in[0,T])$-stopping time with discrete values such that $\beta+\delta_0\leq T$.  It is easy to see that
	\begin{equation}
	\EE [|X_{\beta+\delta}^{1,N}-X_{\beta}^{1,N}|^2]\leq \delta\int_0^{T} \EE[|V_s^{1,N}|^2]ds\leq C\delta\,,
	\end{equation}
	where $C>0$ is independent of $N$ by \eqref{unibound}. Furthermore, following similar arguments as in \eqref{es1}-\eqref{es2}, one obtains
	\begin{equation}
	\EE [|V_{\beta+\delta}^{1,N}-V_{\beta}^{1,N}|^2]\leq C(\delta+\delta^{1/2})\,.
	\end{equation}
	Hence $(Con 2)$ is verified.
	\end{proof}

For any $\varphi\in \mc{C}_c^2(\R^d\times \R^d)$,   define a functional on $\mc{P}(\CR\times \CR)$ as following
\begin{align}
F_{\varphi}(f)&:=\la \varphi(\textbf{x}_t,\textbf{v}_t),f(d\textbf{x},d\textbf{v})\ra-\la \varphi(\textbf{x}_0,\textbf{v}_0),f(d\textbf{x},d\textbf{v}) \ra 
+\int_0^t\la \textbf{v}_s\cdot\nabla_x\varphi,f(d\textbf{x},d\textbf{v})\ra ds
\nn\\
&\quad-\frac{\gamma}{m}\int_0^t\la \textbf{v}_s\cdot\nabla_v\varphi,f(d\textbf{x},d\textbf{v})\ra ds+\frac{\lambda}{m}\int_0^t\la (\textbf{x}_s-X_\alpha(\rho_s))\cdot \nabla_v\varphi, f(d\textbf{x},d\textbf{v})\ra ds\notag\\
&\quad -\frac{\sigma^2}{2m^2}\int_0^t\sum_{k=1}^{d}\la (\textbf{x}_s-X_\alpha(\rho_s))_k^2\frac{\partial^2\varphi}{\partial v_k^2},f (d\textbf{x},d\textbf{v})\ra ds
\notag\\
&=\la \varphi(x,v) ,f_t (dx,dv)\ra-\la \varphi(x,v),f_0(dx,dv) \ra +\int_0^t\la v\cdot\nabla_x\varphi,f_s(dx,dv)\ra ds
\nn\\
&\quad-\frac{\gamma}{m}\int_0^t\la v\cdot\nabla_v\varphi,f_s(dx,dv)\ra ds+\frac{\lambda}{m}\int_0^t\la (x-X_\alpha(\rho_s))\cdot \nabla_v\varphi, f_s (dx,dv)\ra ds\notag\\
&\quad-\frac{\sigma^2}{2m^2}\int_0^t\sum_{k=1}^{d}\la (x-X_\alpha(\rho_s))_k^2\frac{\partial^2\varphi}{\partial v_k^2},f_s (dx,dv)\ra ds \,,
\end{align}
for all $f\in \mc{P}(\CR\times \CR)$ and $\textbf{x},\textbf{v}\in \CR$, where $\rho_s(x)=\int_{ \RR^d }f_s(x,dv)$. Then similar to Proposition \ref{prop}, one can easily prove that
\begin{equation}
\EE[|F_{\varphi}(f^N)|^2]\leq \frac{C}{N}\,.
\end{equation}
Finally, following  similar arguments as in Theorem \ref{thmmean}, there exists a subsequence of $\{f^N\}_{N\geq 2}$ converging in law to a deterministic measure $f\in\mc{P}(\CR\times\CR)$, which is the unique weak solution to the following PDE
\begin{equation}\label{PSO eq}
\partial_{t} f_t+v \cdot \nabla_{x} f_t=\nabla_{v} \cdot\left(\frac{\gamma}{m} v f_t-\frac{\lambda}{m}\left(x-{X}^{\alpha}(\rho_t)\right) f_t+\frac{\sigma^{2}}{2 m^{2}} D\left(x-{X}^{\alpha}(\rho_t)\right)^{2} \nabla_{v} f_t\right)\,,
\end{equation}
where $\rho_t(x)=\int_{ \RR^d }f_t(x,dv)$.  This can be summarized in the following theorem
\begin{thm}
	Let $\TE$ satisfy Assumption \ref{asum} and $f_0\in \mc{P}_4(\RR^d\times\R^d)$. For any $N\geq 2$, we assume that $\{(X_t^{i,N},V_t^{i,N})_{t\in[0,T]}\}_{i=1}^N$ is the unique solution to the particle system \eqref{PSO} with $f_0^{\otimes N}$-distributed initial data $\{X_0^{i,N},V_0^{i,N}\}_{i=1}^N$.  Then the limit (denoted by $f$) of the sequence of the empirical measure $f^N=\frac{1}{N}\sum_{i=1}^N\delta_{(X^{i,N},V^{i,N})}$ exists. Moreover, $f$ is deterministic and it is the unique weak solution to PDE PDE \eqref{PSO eq}.
\end{thm}

\section*{Appendix}
\begin{thm}\label{thmapp}
	For any $T>0$, let $b\in \mc{C}([0,T];\RR^d)$ and $\mu_0\in \mc{P}_2(\RR^d)$. Then the following linear PDE
	\begin{equation}\label{LPDE}
	\partial_t\mu_t=\frac{\sigma^2}{2}\sum_{k=1}^d\frac{\partial^2}{\partial{x_k}^2}((x-b_t)_k^2\mu_t)-\lambda\nabla
	\cdot((x-b_t)\mu_t)\,,
	\end{equation}
	has a unique weak solution $\mu\in \mc{C}([0,T];\mc{P}_2(\RR^d))$
\end{thm}
\begin{proof}[Sketch of the proof]
	The existence is obvious, which can be obtained as the law of the solution to the associated linear SDE. To show  the uniqueness we can follow a duality argument. 
	
	For each $t_0 \in(0, T]$ and compactly supported smooth function $\psi$ (i.e., $\psi\in\mathcal C^{\infty}_c (\RR^d)$), we consider the following backward  PDE
	\begin{align}
	\partial_t h_t=-\frac{\sigma^2}{2}\sum_{k=1}^d (x-b_t)_k^2\frac{\partial^2}{\partial{x_k}^2} h_t-\lambda  (x-b_t)\cdot \nabla h_t,\quad(t,x)\in[0,t_0]\times \RR^d;\quad 
	h_{t_0}=\psi\,. \label{backward-PDE}
	\end{align}
	which admits a classical solution $h\in\mc{C}^1([0,t_0],\mc{C}^2(\RR^d))$.   Indeed, we can explicitly construct a solution 
	\begin{equation}
	h_t(x)=\EE[\psi(X_{t_0}^{t,x})]\quad t\in[0,t_0], \label{solution-backward}
	\end{equation}
	where $(X_s^{t,x})_{ 0\leq t\leq s\leq t_0}$ is the strong solution to the following linear SDE
	\begin{equation}
	d X_s^{t,x}=\lambda (X_s^{t,x}-b_s)ds+\sigma D(X_s^{t,x}-b_s)dB_s,\quad X_t^{t,x}=x, \label{SDE-appendix}
	\end{equation}
	with $D(y)=\diag(y_1,\cdots,y_d)$ for $y\in\RR^d$ and $B=(B^1,\dots,B^d)$ being a d-dimensional Wiener process. We can first check the regularity. For each $(t,x)\in[0,t_0]\times\RR^d$, the chain rule gives
	$$
	\nabla_kh_t(x)= \EE\left[\nabla_k\psi(X_{t_0}^{t,x}) \nabla_k  (X_{t_0}^{t,x})^k \right],\quad k=1,\cdots,d.
	$$
	Note that $\nabla_k (X_{t_0}^{t,x})^{k'}=0$ when $k'\neq k$ and  that $\nabla_k  (X^{t,x})^k$ is a Geometric Brownian motion satisfying SDE (c.f. \cite[Theorem 4.2]{metivier1982pathwise})
	\begin{equation*}
	d \nabla_k  (X_s^{t,x})^k=\lambda \nabla_k  (X_s^{t,x})^k\, ds+\sigma  \nabla_k  (X_s^{t,x})^k \,dB^k_s,\quad \nabla_k  (X_t^{t,x})^k=1.
	\end{equation*}
	This gives $\nabla_k  (X_s^{t,x})^k=\exp\{\lambda(s-t)-\frac{\sigma^2(s-t)}{2}-\sigma(B^k_s-B^k_t)\}$. Accordingly,  we may obtain the time-space continuity of $\nabla_kh_t(x)$ and in particular, there holds the following uniform boundness
	$$
	\sup_{(t,x)\in[0,t_0]\times\RR^d} \left| \nabla_kh_t(x) \right|
	\leq  C \EE\left[ \left| \nabla_k  (X_{t_0}^{t,x})^k \right| \right]
	\leq C e^{\lambda T}<\infty,
	\quad k=1,\cdots,d,
	$$
	where $C>0$ is a constant depending on $\psi$.
	Analogously, we may derive the uniform boundness of $\nabla^2 h$ and even of $\nabla^3h$ together with associated time-space continuity. On the other hand, for $0\leq t<t+\delta<t_0$, the flow property of solution to SDE \eqref{SDE-appendix} implies $X_s^{t,x}=X_s^{t+\delta,X_{t+\delta}^{t,x}}$ for $t+\delta<s\leq t_0$ and thus,
	\begin{align*}
	\frac{h_{t+\delta}(x)-h_t(x)}{\delta}
	&=\frac{1}{\delta} \EE\left[
	\psi(X_{t_0}^{t+\delta,x}) -\psi(X_{t_0}^{t,x})
	\right]
	\\
	&=\frac{1}{\delta} \EE\left[
	\psi(X_{t_0}^{t+\delta,x}) -\psi(X_{t_0}^{t+\delta,X_{t+\delta}^{t,x}})
	\right]
	\\
	&= \frac{1}{\delta} \EE\left[
	h_{t+\delta}( x)-h_{t+\delta}(X_{t+\delta}^{t,x})
	\right]
	\\
	&=-\frac{1}{\delta} \EE\left[
	\int_t^{t+\delta}
	\left( \frac{\sigma^2}{2}\sum_{k=1}^d (X_{s}^{t,x}-b_s)_k^2\frac{\partial^2}{\partial{x_k}^2}  h_{t+\delta}(X_{s}^{t,x})+\lambda  (X_{s}^{t,x}-b_s)\cdot \nabla h_{t+\delta}(X_{s}^{t,x}) \right)
	\,ds 
	\right].
	\end{align*}
	Through a simple limiting procedure we may get the time-differentiability of $h_t(x)$ and further verify that  the defined $h_t(x)$ is a classical solution of PDE \eqref{backward-PDE}\footnote{An application of It\^o-Doeblin formula also gives the uniqueness, whereas the existence is sufficient here.}.
	
	Suppose that $\mu^1$ and $\mu^2$ are two weak solutions of \eqref{LPDE} with the same initial condition $\mu^1_0=\mu^2_0$. Put $\delta\mu=\mu^1-\mu^2$. Using the above defined solution $h$ to the backward PDE \eqref{backward-PDE} as a test function, we have
	\begin{align}
	&   \la h_{t_0}(x),\delta\mu_{t_0}(dx) \ra  
	\notag\\
	=&\int_0^{t_0}\la \partial_sh_s(x),\delta\mu_s(dx)\ra ds+\int_0^{t_0}\la  \frac{\sigma^2}{2}\sum_{k=1}^d (x-b_s)_k^2\frac{\partial^2}{\partial{x_k}^2}  h_s(x),\delta\mu_s(dx)\ra ds
	+\int_0^{t_0}\la \lambda(x-b_s)\cdot\nabla h_s(x),\delta\mu_s(dx)\ra ds\notag\\
	=&\int_0^{t_0}\la \partial_sh_s(x),\delta\mu_s(dx)\ra ds+\int_0^{t_0}\la -\partial_sh_s(x),\delta\mu_s(dx)\ra ds
	\notag\\
	=&0,
	\end{align}
	which gives $\int_{\R^d} \psi(x)\delta\mu_{t_0}(dx)=0$ for arbitrary $\psi\in\mc{C}_c^\infty(\RR^d)$. This implies that $\delta\mu_{t_0}=0$, which yields the uniqueness by the arbitrariness of $t_0$.
\end{proof}

\bibliographystyle{amsxport}
\bibliography{MFLCBO}


\end{document}